\documentclass[reqno,12pt]{amsart}
\usepackage{amscd,amsfonts,amssymb}
\numberwithin{equation}{section}
\newtheorem{Theorem}{Theorem}[section]
\newtheorem{Lemma}{Lemma}[section]

\theoremstyle{definition}
\newtheorem{Definition}{Definition}[section]
\theoremstyle{remark}
\newtheorem{Remark}{Remark}[section]
\newtheorem{Example}{Example}[section]

\newcommand{\esssup}{\mathop{\rm ess \, sup}\limits}
\newcommand{\mes}{\mathop{\rm mes}\nolimits}

\author{Andrej A. Kon'kov}
\address{Department of Differential Equations,
Faculty of Mechanics and Mathematics,
Mo\-s\-cow Lo\-mo\-no\-sov State University,
Vorobyovy Gory,
119992 Moscow, Russia}
\email{konkov@mech.math.msu.su}
\title[]{On the behavior of Kneser solutions of nonlinear ordinary differential equations}
\thanks{The research was supported by RFBR, grant 11-01-12018-ofi-m-2011.}
\date{}
\begin{document}

\begin{abstract}
We obtain priory estimates and sufficient conditions for Kneser solutions of ordinary differential
equations to vanish in a neighborhood of infinity.
\end{abstract}

\maketitle

\section{Introduction}

We study solutions of the differential equations
\begin{equation}
	w^{(m)}
	=
	Q (r, w, \ldots, w^{(m-1)}),
	\quad
	r \ge a,
	\label{1.1}
\end{equation}
of order $m \ge 2$ satisfying the conditions
\begin{equation}
	(-1)^i w^{(i)} (r)
	\ge
	0,
	\quad
	r \ge a,
	\quad
	i = 0, \ldots, m - 1,
	\label{1.2}
\end{equation}
where $Q$ belongs to the Caratheodory class 
$
    K_{loc}
    \left(
        [a,\infty) \times {\mathbb R}^m
    \right),
$
$a > 0$~\cite{KCbook}. 
Throughout the paper, it is assumed that 
\begin{equation}
	(-1)^m
	Q (r, t_0, \ldots, t_{m-1})
	\ge
	q (r)
	h (t_0)
	-
	\sum_{i = 1}^{m - 1}
	b_i (r) |t_i|
	\label{1.3}
\end{equation}
on the set 
$
	\{
		(r,t_0,\ldots,t_{m-1})
		:
		r \ge a,
		\;
		t_0 > 0,
		\;
		(-1)^i t_i \ge 0,
		\:
		i = 1, \ldots, m-1
	\},
$
where
$q : [a,\infty) \to [0,\infty)$
and
$b_i : [a,\infty) \to [0,\infty)$,
$i = 1, \ldots, m - 1$,
belong to the space
$L_{\infty, loc} ([a,\infty))$
and
$h : (0,\infty) \to (0,\infty)$
is a continuous function.

As an example of~\eqref{1.1}, one can take the equation
$$
	w^{(m)} 
	+ 
	z_{m - 1} (r) w^{(m - 1)} 
	+ 
	\ldots 
	+ 
	z_1 (r) w' 
	= 
	z (r, w)
$$
in which the coefficients of lower derivatives and the right-hand side are continuous functions with
${(-1)^m z (r, t)} \ge {q (r) h (t)}$ for all $r \ge a$ and $t > 0$.
Inequality~\eqref{1.3} should obviously be fulfilled, if we put
$$
	Q (r, t_0, \ldots, t_{m - 1})
	=
	z (r, t_0)
	-
	\sum_{i = 1}^{m - 1}
	z_i (r)
	t_i
$$
and
$b_i = |z_i|$, $i = 1, \ldots, m - 1$.

As is customary, a function $w : [a, \infty) \to {\mathbb R}$
is called a solution of~\eqref{1.1}, \eqref{1.2}
if its derivatives $w^{(i)}$, $i = 0, \ldots, m - 1$,
are locally absolutely continuous on the interval $[a, \infty)$,
equation~\eqref{1.1} holds for almost all $r \in [a, \infty)$,
and conditions~\eqref{1.2} hold for all $r \in [a, \infty)$.

By a non-trivial solution of~\eqref{1.1}, \eqref{1.2} we mean a solution
that does not vanish on the whole interval $[a, \infty)$.

\begin{Definition}[\cite{KCbook}]
A non-trivial solution of~\eqref{1.1}, \eqref{1.2} is singular of the first kind
if it vanishes in a neighborhood of infinity;
otherwise this solution is called regular (proper).
\end{Definition}

In the literature, solutions of problem~\eqref{1.1}, \eqref{1.2} are also known as Kneser solutions.
Starting from the pioneering paper of A.~Kneser~\cite{Kneser},
they attract the attention of many mathematicians~[1--10].
Our aim is to obtain priory estimates and sufficient conditions 
for any solution of~\eqref{1.1}, \eqref{1.2} to be singular of the first kind.
In particular, we generalize results of~\cite{meJMS}, 
where the case of $b_1 = \ldots = b_{m-1} = 0$ was considered.

\section{Main Results}

We denote
$$
	g (t)
	=
	\inf_{
		(t / \theta, \theta t)
	}
	h,
	\quad
	t \in (0, \infty),
$$
$$
	f (r)
	=
	\frac{
		q (r)
	}{
		1
		+
		\sum\limits_{i = 1}^{m - 1}
		r^{m - i}
		\esssup_{
			(r / \sigma, r \sigma)
			\cap
			[a, \infty)
		}
		b_i
	},
	\quad
	r \in [a, \infty)
$$
and
$$
	\mu (r)
	=
	1
	+
	r^m
	\esssup_{
		(r / \sigma, r)
		\cap
		[a, \infty)
	}
	f,
	\quad
	r \in [a, \infty),
$$
where $\theta > 1$ and $\sigma > 1$ are some real numbers which can be chosen arbitrary.

\begin{Theorem}\label{T2.1}
Let
\begin{equation}
	\int_0^1
	g^{- 1 / m} (t)
	t^{1 / m - 1}
	\,
	dt
	<
	\infty
	\label{T2.1.1}
\end{equation}
and
\begin{equation}
	\int_a^\infty
	\xi^{m - 1}
	f (\xi)
	\mu^{1 / m - 1} (\xi)
	\,
	d \xi
	=
	\infty.
	\label{T2.1.2}
\end{equation}
Then any non-trivial solution of~\eqref{1.1}, \eqref{1.2} is singular of the first kind.
\end{Theorem}

\begin{Theorem}\label{T2.2}
Suppose that condition~\eqref{T2.1.1} is valid,
\begin{equation}
	\int_a^\infty
	\xi^{m - 1}
	f (\xi)
	\,
	d \xi
	=
	\infty
	\label{T2.2.1}
\end{equation}
and, moreover,
\begin{equation}
	\limsup_{r \to \infty}
	\frac{
		r^m f (r)
	}{
		\int_a^r
		\xi^{m - 1}
		f (\xi)
		\,
		d \xi
	}
	<
	\infty.
	\label{T2.2.2}
\end{equation}
Then any non-trivial solution of~\eqref{1.1}, \eqref{1.2} is singular of the first kind.
\end{Theorem}

Theorems~\ref{T2.1} and~\ref{T2.2} are proved in Section~\ref{proofOfTheorems}.
Now, we demonstrate their exactness.

\begin{Example}\label{E2.1}
Consider the problem
\begin{align}
	&
	w'' + b (r) w' = p (r) w^\lambda,
	\quad
	r \ge a,
	\label{E2.1.1}
	\\
	&
	(- 1)^i w^{(i)} (r) \ge 0,
	\quad
	r \ge a,
	\quad
	i = 0, 1,
	\label{E2.1.2}
\end{align}
where $\lambda < 1$ and, moreover,
$b : [a, \infty) \to {\mathbb R}$ 
and
$p : [a, \infty) \to [0, \infty)$
are locally bounded measurable functions such that
\begin{equation}
	|b (r)| \le B r^s,
	\;
	B = const > 0,
	\label{E2.1.3}
\end{equation}
for all sufficiently large $r$ and
\begin{equation}
	p (r) \sim r^l
	\quad
	\mbox{as } r \to \infty,
	\label{E2.1.4}
\end{equation}
i.e. 
$$
	c_1 r^l \le p (r) \le c_2 r^l
$$
with some constants $c_1 > 0$ and $ c_2 > 0$ for almost all $r$ in a neighborhood of infinity.

At first, let $s > -1$.
By Theorem~\ref{T2.2}, if
\begin{equation}
	l \ge s - 1,
	\label{E2.1.5}
\end{equation}
then any non-trivial solution of~\eqref{E2.1.1}, \eqref{E2.1.2} is singular of the first kind.
At the same time, in case of $l < s - 1$, 
it does not present any particular problem to verify that
$$
	w (r)
	=
	r^{(l - s + 1) / (1 - \lambda)}
$$
is a regular solution of~\eqref{E2.1.1}, \eqref{E2.1.2},
where $b (r) = - r^s$ and $p$ is a non-negative continuous function
satisfying relation~\eqref{E2.1.4}.
Therefore, condition~\eqref{E2.1.5} is exact.

Now, assume that $s \le - 1$. If
\begin{equation}
	l \ge - 2,
	\label{E2.1.6}
\end{equation}
then Theorem~\ref{T2.2} implies that 
any non-trivial solution of~\eqref{E2.1.1}, \eqref{E2.1.2} is singular of the first kind.
This condition is exact too. Namely, if $l < - 2$, then, putting
$$
	w (r)
	=
	r^{(l + 2) / (1 - \lambda)},
$$
we obviously obtain a regular solution of~\eqref{E2.1.1}, \eqref{E2.1.2}, 
where $b \equiv 0$ and $p$ is a non-negative continuous function 
for which~\eqref{E2.1.4} holds.
\end{Example}

\begin{Example}\label{E2.2}
In equation~\eqref{E2.1.1}, let
$b : [a, \infty) \to {\mathbb R}$ 
and
$p : [a, \infty) \to [0, \infty)$
are locally bounded measurable functions such that~\eqref{E2.1.3} is valid with $s > - 1$ 
and, moreover,
\begin{equation}
	p (r)
	\sim
	r^{s - 1}
	\log^\nu r
	\quad
	\mbox{as }
	r \to \infty.
	\label{E2.2.1}
\end{equation}
In the case of $\nu = 0$, the last relation obviously takes the form~\eqref{E2.1.4} 
with the critical exponent $l = s - 1$.
Also assume that $\lambda < 1$.

According to Theorem~\ref{T2.2}, if
\begin{equation}
	\nu \ge - 1,
	\label{E2.2.2}
\end{equation}
then any non-trivial solution of~\eqref{E2.1.1}, \eqref{E2.1.2} is singular of the first kind.
In so doing, if $\nu < - 1$, then there exists a real number $\varepsilon > 0$ such that
$$
	w (r)
	=
	\log^{(\nu + 1) / (1 - \lambda)}
	(r / \varepsilon)
$$
is a regular solution of~\eqref{E2.1.1}, \eqref{E2.1.2},
where
$b (r) = - r^s$
and
$p  \in C ([a, \infty))$
is a non-negative function satisfying relation~\eqref{E2.2.1}.
This demonstrates the exactness of~\eqref{E2.2.2}.

Now, let~\eqref{E2.1.3} is valid with $s \le -1$.
We examine the critical exponent $l = -2$ in formula~\eqref{E2.1.4}.
Assume that
\begin{equation}
	p (r)
	\sim
	r^{- 2}
	\log^\gamma
	r
	\quad
	\mbox{as }
	r \to \infty.
	\label{E2.2.3}
\end{equation}
By Theorem~\ref{T2.2}, if
$$
	\gamma \ge - 1,
$$
then any non-trivial solution of~\eqref{E2.1.1}, \eqref{E2.1.2} is singular of the first kind.
The above condition is exact.
Really, in the case of $\gamma < -1$, it can be verified that
$$
	w (r)
	=
	\log^{(\gamma + 1) / (1 - \lambda)} 
	(r / \varepsilon)
$$
is a regular solution of~\eqref{E2.1.1}, \eqref{E2.1.2} for some sufficiently small $\varepsilon > 0$,
where $b \equiv 0$ and $p  \in C ([a, \infty))$
is a non-negative function for which~\eqref{E2.2.3} holds.
\end{Example}

\begin{Example}\label{E2.3}
Consider the equation
\begin{equation}
	w'' + b (r) w' = p (r) w \ln^\lambda \left( 1 + \frac{1}{w}  \right),
	\quad
	r \ge a,
	\label{E2.3.1}
\end{equation}
where $\lambda > 2$ and, moreover,
$b : [a, \infty) \to {\mathbb R}$ 
and
$p : [a, \infty) \to [0, \infty)$
are locally bounded measurable functions such that~\eqref{E2.1.3} and~\eqref{E2.1.4} are fulfilled.

If $s > - 1$, then in accordance with Theorem~\ref{T2.2} inequality~\eqref{E2.1.5} 
guarantees that any non-trivial solution of~\eqref{E2.3.1}, \eqref{E2.1.2} 
is singular of the first kind.
Assume that~\eqref{E2.1.5} does not hold or, in other words, $l < s - 1$.
Then, putting
$$
	w (r)
	=
	1 + r^{l - s + 1},
$$
we obtain a regular solution of~\eqref{E2.3.1}, \eqref{E2.1.2}, where $b (r) = - r^s$
and $p$ is a non-negative continuous function satisfying relation~\eqref{E2.1.4}.

For $s \le - 1$, by Theorem~\ref{T2.2}, any non-trivial solution of~\eqref{E2.3.1}, \eqref{E2.1.2} 
is singular of the first kind if inequality~\eqref{E2.1.6} is valid.
In turn, if~\eqref{E2.1.6} is not valid, then 
$$
	w (r)
	=
	1 + r^{l + 2}
$$
is a regular solution of~\eqref{E2.3.1}, \eqref{E2.1.2},
where $b \equiv 0$ and $p$ is a non-negative continuous function 
for which~\eqref{E2.1.4} holds.
\end{Example}

\begin{Theorem}\label{T2.3}
In the hypotheses of Theorem~$\ref{T2.2}$, let the condition
\begin{equation}
	\int_0^1
	g^{- 1 / m} (t)
	t^{1 / m - 1}
	\,
	dt
	=
	\infty
	\label{T2.3.1}
\end{equation}
be fulfilled instead of~\eqref{T2.1.1}. 
Then any solution of~\eqref{1.1}, \eqref{1.2} satisfies the estimate
\begin{equation}
	w (r)
	\le
	G_0^{-1}
	\left(
		C
		\left(
			\int_a^r
			\xi^{m - 1}
			f (\xi)
			\,
			d\xi
		\right)^{1 / m}
	\right)
	\label{T2.3.2}
\end{equation}
for all sufficiently large $r$,
where $G_0^{-1}$ is the function inverse to
$$
	G_0 (\xi)
	=
	\int_\xi^1
	g^{- 1 / m} (t)
	t^{1 / m - 1}
	\,
	dt,
$$
and the constant $C > 0$ depends only on $m$, $\theta$, $\sigma$,
and on the value of the limit in the left-hand side of~\eqref{T2.2.2}.
\end{Theorem}

\begin{Theorem}\label{T2.4}
Let~\eqref{T2.1.1} be valid,
$$
	\int_a^\infty
	\xi^{m - 1}
	f (\xi)
	\,
	d\xi
	<
	\infty
$$
and, moreover,
\begin{equation}
	\limsup_{r \to \infty}
	\frac{
		r^m f (r)
	}{
		\int_r^\infty
		\xi^{m - 1}
		f (\xi)
		\,
		d\xi
	}
	<
	\infty.
	\label{T2.4.1}
\end{equation}
Then any regular solution of~\eqref{1.1}, \eqref{1.2} satisfies the estimate
\begin{equation}
	w (r)
	\ge
	G_\infty^{-1}
	\left(
		C
		\left(
			\int_r^\infty
			\xi^{m - 1}
			f (\xi)
			\,
			d\xi
		\right)^{1 / m}
	\right)
	\label{T2.4.2}
\end{equation}
for all sufficiently large $r$,
where $G_\infty^{-1}$ is the function inverse to
$$
	G_\infty (\xi)
	=
	\int_0^\xi
	g^{- 1/ m} (t)
	t^{1/ m - 1}
	\,
	dt,
$$
and the constant $C > 0$ depends only on $m$, $\theta$, $\sigma$, 
and on the value of the limit in the left-hand side of~\eqref{T2.4.1}.
\end{Theorem}

Theorems~\ref{T2.3} and \ref{T2.4} are proved in Section~\ref{proofOfTheorems}.

\begin{Example}\label{E2.4}
Consider equation~\eqref{E2.1.1}, where
$b : [a, \infty) \to {\mathbb R}$ 
and
$p : [a, \infty) \to [0, \infty)$
are locally bounded measurable functions such that~\eqref{E2.1.3} and~\eqref{E2.1.4} are valid.

At first, let $s > -1$. If $\lambda > 1$ and $l > s - 1$, then
in accordance with Theorem~\ref{T2.3} any solution of~\eqref{E2.1.1}, \eqref{E2.1.2} 
satisfies the estimate
$$
	w (r)
	\le
	C r^{(l - s + 1) / (1 - \lambda)}
$$
for all sufficiently large $r$, where the constant $C > 0$ does not depend on $w$.
In the case that $\lambda < 1$ and $l < s - 1$, using Theorem~\ref{T2.4}, we obtain
$$
	w (r)
	\ge
	C r^{(l - s + 1) / (1 - \lambda)}
$$
for any regular solution of~\eqref{E2.1.1}, \eqref{E2.1.2}, 
where $r$ runs through a neighborhood of infinity and $C > 0$ is a constant independent of $w$.

Now, assume that $s \le -1$. By Theorem~\ref{T2.3}, if $\lambda > 1$ and $l > -2$, 
then any solution of~\eqref{E2.1.1}, \eqref{E2.1.2} satisfies the estimate
$$
	w (r)
	\le
	C r^{(l + 2) / (1 - \lambda)}
$$
for all sufficiently large $r$, where the constant $C > 0$ does not depends on $w$.
In turn, if $\lambda < 1$ and $l < -2$, then Theorem~\ref{T2.4} implies the inequality
$$
	w (r)
	\ge
	C r^{(l + 2) / (1 - \lambda)}
$$
for any regular solution of~\eqref{E2.1.1}, \eqref{E2.1.2},
where $r$ runs through a neighborhood of infinity and $C > 0$ is a constant independent of $w$.
\end{Example}

\begin{Example}\label{E2.5}
In~\eqref{E2.1.1}, let the locally bounded measurable functions 
$b : [a, \infty) \to {\mathbb R}$ 
and
$p : [a, \infty) \to [0, \infty)$
satisfy relations~\eqref{E2.1.3} and~\eqref{E2.2.1}, where $s > -1$.
If $\lambda > 1$ and $\nu > -1$, then in accordance with Theorem~\ref{E2.3} we have
$$
	w (r)
	\le
	C \log^{(\nu + 1) / (1 - \lambda)} r
$$
for any solution of~\eqref{E2.1.1}, \eqref{E2.1.2},
where $r$ runs through a neighborhood of infinity and $C > 0$ is a constant independent of $w$.
In the case that $\lambda < 1$ and $\nu < -1$, by Theorem~\ref{E2.4},
any regular solution of~\eqref{E2.1.1}, \eqref{E2.1.2} satisfies the estimate
$$
	w (r)
	\ge
	C \log^{(\nu + 1) / (1 - \lambda)} r
$$
for all sufficiently large $r$, where the constant $C > 0$ does not depend on $w$.

Now, assume that $s \le -1$ in~\eqref{E2.1.3} and, moreover, 
relation~\eqref{E2.2.3} is fulfilled instead of~\eqref{E2.2.1}.
By Theorem~\ref{E2.3}, if $\lambda > 1$ and $\gamma > -1$, then
$$
	w (r)
	\le
	C \log^{(\gamma + 1) / (1 - \lambda)} r
$$
for any solution of~\eqref{E2.1.1}, \eqref{E2.1.2},
where $r$ runs through a neighborhood of infinity and $C > 0$ is a constant independent of $w$.
In turn, if $\lambda < 1$ and $\gamma < -1$, then in accordance with Theorem~\ref{E2.4} 
one can claim that
$$
	w (r)
	\ge
	C \log^{(\gamma + 1) / (1 - \lambda)} r
$$
for any regular solution of~\eqref{E2.1.1}, \eqref{E2.1.2},
where $r$ runs through a neighborhood of infinity and $C > 0$ is a constant independent of $w$.
\end{Example}

It does not present any particular problem to show that all estimates 
given in Examples~\ref{E2.4} and~\ref{E2.5} are exact.

\section{Proof of Theorems~\ref{T2.1}--\ref{T2.4}}\label{proofOfTheorems}

Agree on the following notation.
In this section, by $c$ we denote various positive constants 
that can depend only on $m$, $\theta$, and $\sigma$.
For Lemma~\ref{L3.7} and Theorem~\ref{T2.3}, 
these constants can also depend on the value of the limit in the left-hand side of~\eqref{T2.2.2}.
Analogously, in the case of Theorem~\ref{T2.4}, the constants $c$ can depend
on the value of the limit in the left-hand side of~\eqref{T2.4.1}.
We put
$$
	\eta (t)
	=
	\inf_{
		(\theta^{- 1 / 2} t, \theta^{1 / 2} t)
	}
	h,
	\quad
	t \in (0, \infty)
$$
and
$$
	\varphi (\xi)
	=
	\xi^m
	\esssup_{
		(\xi / \sigma, \xi)
		\cap
		[a, \infty)
	}
	f,
	\quad
	\xi \in [a, \infty).
$$

\begin{Lemma}\label{L3.1}
Let $w$ be a solution of~\eqref{1.1}, \eqref{1.2} and, moreover, $a \le r_1 < r_2$ be real numbers
such that $\sigma r_1 \ge r_2$ and $\theta^{1 / 2} w (r_2) \ge w (r_1) > 0$.
Then
$$
	w (r_1)
	-
	w (r_2)
	\ge
	c
	\sup_{
		[w (r_2), w (r_1)]
	}
	\eta
	\int_{
		r_1
	}^{
		r_2
	}
	(\xi - r_1)^{m - 1}
	f (\xi)
	\,
	d \xi.
$$
\end{Lemma}

\begin{proof}
Integrating by parts, we obtain
\begin{align*}
	w (r_1)
	-
	w (r_2)
	=
	&
	\sum_{1 \le k \le i - 1}
	\frac{
		(-1)^k w^{(k)} (r_2)
		(r_2 - r_1)^k
	}{
		k!
	}
	\\
	&
	+
	\frac{1}{(i - 1)!}
	\int_{
		r_1
	}^{
		r_2
	}
	(\xi - r_1)^{i - 1}
	(-1)^i
	w^{(i)} (\xi)
	\,
	d\xi,
	\quad
	i = 1, \ldots, m.
\end{align*}
By~\eqref{1.2}, this implies the inequality
\begin{equation}
	w (r_1)
	-
	w (r_2)
	\ge
	\frac{1}{(i - 1)!}
	\int_{
		r_1
	}^{
		r_2
	}
	(\xi - r_1)^{i - 1}
	(-1)^i
	w^{(i)} (\xi)
	\,
	d \xi,
	\quad
	i = 1, \ldots, m.
	\label{PL3.1.1}
\end{equation}
We denote 
$$
	\Omega_i 
	= 
	\left\{ 
		\xi \in [r_1, r_2]
		:
		(-1)^i w^{(i)} (\xi) 
		b_i (\xi)
		\ge 
		\frac{1}{m}
		q (\xi) h (w)
	\right\},
	\quad
	i = 1, \ldots, m - 1.
$$
In addition, let
$$
	\Omega_m
	= 
	\left\{ 
		\xi \in [r_1, r_2]
		:
		(-1)^m w^{(m)} (\xi) 
		\ge 
		\frac{1}{m}
		q (\xi) h (w)
	\right\}.
$$
From~\eqref{1.1}--\eqref{1.3}, it follows that
$$
	\mes
	{}
	[r_1, r_2]
	\setminus
	\cup_{i=1}^m
	\Omega_i
	=
	0.
$$
For all $i \in \{ 1, \ldots, m\}$ we have
\begin{equation}
	(-1)^i w^{(i)} (\xi) 
	\ge 
	\frac{1}{m}
	r^{m - i}
	f (\xi) h (w)
	\label{PL3.1.5}
\end{equation}
almost everywhere on $\Omega_i$.
Hence, condition~\eqref{1.2} allows us to assert that
$$
	(-1)^i w^{(i)} (\xi) 
	\ge 
	\frac{1}{m}
	\chi_{\Omega_i} (\xi)
	\xi^{m - i}
	f (\xi) h (w)
$$
for almost all $\xi \in [r_1, r_2]$ and for all $i \in \{ 1, \ldots, m - 1 \}$,
where $\chi_{\Omega_i}$ is the characteristic function of the set $\Omega_i$, i.e.
$$
	\chi_{\Omega_i} (\xi)
	=
	\left\{
		\begin{array}{ll}
			1,
			&
			\xi
			\in
			\Omega_i,
			\\
			0,
			&
			\xi
			\not\in
			\Omega_i.
		\end{array}
	\right.
$$
Combining this with~\eqref{PL3.1.1}, we obtain
\begin{equation}
	w (r_1)
	-
	w (r_2)
	\ge
	c
	\int_{\Omega_i}
	(\xi - r_1)^{m - 1}
	f (\xi)
	h (w)
	\,
	d\xi,
	\quad
	i = 1, \ldots, m - 1.
	\label{PL3.1.2}
\end{equation}

Let us also establish the validity of the inequality
\begin{equation}
	w (r_1)
	-
	w (r_2)
	\ge
	c
	\int_{\Omega_m}
	(\xi - r_1)^{m - 1}
	f (\xi)
	h (w)
	\,
	d \xi.
	\label{PL3.1.3}
\end{equation}
Really, taking~\eqref{PL3.1.1} into account, we have
\begin{align}
	w (r_1)
	-
	w (r_2)
	\ge
	{}
	&
	\frac{1}{(m - 1)!}
	\int_{
		\Omega^{+}
	}
	(\xi - r_1)^{m - 1}
	|w^{(m)} (\xi)|
	\,
	d \xi
	\nonumber
	\\
	&
	-
	\frac{1}{(m - 1)!}
	\int_{
		\Omega^{-}
	}
	(\xi - r_1)^{m - 1}
	|w^{(m)} (\xi)|
	\,
	d \xi,
	\label{PL3.1.4}
\end{align}
where
$
	\Omega^{+} 
	= 
	\{ 
		\xi \in [r_1, r_2]
		:
		(-1)^m w^{(m)} (\xi) \ge 0
	\}
$
and $\Omega^{-} = [r_1, r_2] \setminus \Omega^{+}$.
In the case that
$$
	\frac{1}{2}
	\int_{
		\Omega^{+}
	}
	(\xi - r_1)^{m - 1}
	|w^{(m)} (\xi)|
	\,
	d \xi
	\ge
	\int_{
		\Omega^{-}
	}
	(\xi - r_1)^{m - 1}
	|w^{(m)} (\xi)|
	\,
	d \xi,
$$
formula~\eqref{PL3.1.4} implies the estimate
$$
	w (r_1)
	-
	w (r_2)
	\ge
	\frac{1}{2 (m - 1)!}
	\int_{
		\Omega^{+}
	}
	(\xi - r_1)^{m - 1}
	|w^{(m)} (\xi)|
	\,
	d \xi
$$
from which, taking into account~\eqref{PL3.1.5} and the evident inclusion 
$\Omega_m \subset \Omega^{+}$, we immediately obtain~\eqref{PL3.1.3}.

Now, let
\begin{equation}
	\frac{1}{2}
	\int_{
		\Omega^{+}
	}
	(\xi - r_1)^{m - 1}
	|w^{(m)} (\xi)|
	\,
	d \xi
	<
	\int_{
		\Omega^{-}
	}
	(\xi - r_1)^{m - 1}
	|w^{(m)} (\xi)|
	\,
	d \xi.
	\label{PL3.1.6}
\end{equation}
We put
$$
	\omega_i
	=
	\left\{
		\xi \in \Omega_i
		:
		(-1)^i
		w^{(i)} (\xi)
		b_i (\xi)
		\ge
		\frac{
			1
		}{
			m - 1
		}
		|w^{(m)} (\xi)|
	\right\},
	\quad
	i = 1, \ldots, m - 1.
$$

According to~\eqref{1.1}--\eqref{1.3}, for almost all $\xi \in \Omega^{-}$ 
there exists $i \in \{ 1, \ldots, m - 1 \}$ such that
$$
	(-1)^i w^{(i)} (\xi)
	b_i (\xi)
	\ge 
	\frac{
		|w^{(m)} (\xi)|
		+
		q (\xi) h (w)
	}{
		m - 1
	}.
$$
Consequently,
$$
	\mes
	\Omega^{-}
	\setminus
	\cup_{i=1}^{m-1}
	\omega_i
	=
	0.
$$
For all $i \in \{ 1, \ldots, m - 1 \}$ we obviously have
$$
	(-1)^i
	w^{(i)} (\xi)
	b_i (\xi)
	\ge
	\frac{
		\xi^{m - i}
		|w^{(m)} (\xi)|
	}{
		(m - 1)
		(1 + \xi^{m - i} b_i (\xi))
	}
$$
almost everywhere on $\omega_i$. 
Combining the last inequality with~\eqref{PL3.1.1}, one can conclude that
$$
	w (r_1)
	-
	w (r_2)
	\ge
	\frac{
		c
	}{
		1 
		+ 
		r_2^{m - i}
		\esssup_{
			(r_1, r_2)
		}
		b_i
	}
	\int_{\omega_i}
	(\xi - r_1)^{m - 1}
	|w^{(m)} (\xi)|
	\,
	d \xi,
	\quad
	i = 1, \ldots, m - 1.
$$
This implies the estimate
$$
	w (r_1)
	-
	w (r_2)
	\ge
	\frac{
		c
	}{
		1 
		+ 
		\sum_{i=1}^{m-1}
		r_2^{m - i}
		\esssup_{
			(r_1, r_2)
		}
		b_i
	}
	\int_{\Omega^{-}}
	(\xi - r_1)^{m - 1}
	|w^{(m)} (\xi)|
	\,
	d \xi
$$
from which, using~\eqref{PL3.1.6}, we obtain
\begin{align*}
	w (r_1)
	-
	w (r_2)
	&
	\ge
	\frac{
		c
	}{
		1 
		+ 
		\sum_{i=1}^{m-1}
		r_2^{m - i}
		\esssup_{
			(r_1, r_2)
		}
		b_i
	}
	\int_{\Omega^{+}}
	(\xi - r_1)^{m - 1}
	|w^{(m)} (\xi)|
	\,
	d \xi
	\\
	&
	\ge
	\frac{
		c
	}{
		1 
		+ 
		\sum_{i=1}^{m-1}
		r_2^{m - i}
		\esssup_{
			(r_1, r_2)
		}
		b_i
	}
	\int_{\Omega_m}
	(\xi - r_1)^{m - 1}
	q (\xi) 
	h (w)
	\,
	d \xi,
\end{align*}
whence~\eqref{PL3.1.3} follows at once.

Inequalities~\eqref{PL3.1.2} and~\eqref{PL3.1.3} enable us to assert that
$$
	w (r_1)
	-
	w (r_2)
	\ge
	c
	\int_{r_1}^{r_2}
	(\xi - r_1)^{m - 1}
	f (\xi)
	h (w)
	\,
	d \xi.
$$
Since
$$
	\inf_{
		\xi \in (r_1, r_2)
	}
	h (w (\xi))
	\ge
	\sup_{
		[w (r_2), w (r_1)]
	}
	\eta,
$$
this completes the proof.
\end{proof}

\begin{Lemma}\label{L3.2}
Let
$r_1 < r_2$
and
$0 < \alpha < 1$
be some real numbers, then
$$
	\left(
		\int_{
			r_1
		}^{
			r_2
		}
		\psi (\xi)
		\,
		d \xi
	\right)^\alpha
	\ge
	A
	\int_{
		r_1
	}^{
		r_2
	}
	\psi (\xi)
	\varkappa^{
	\alpha - 1
	}
	(\xi)
	\,
	d\xi
$$
for any non-negative function $\psi \in L ([r_1, r_2])$, where
$$
	\varkappa (\xi)
	=
	\int_{r_1}^\xi
	\psi (\zeta)
	\,
	d \zeta
$$
and $A > 0$ is a constant depending only on $\alpha$.
\end{Lemma}

\begin{Remark}\label{R3.1}
If $\varkappa (\xi) = 0$ for some $\xi \in (r_1, r_2)$,
then $\psi = 0$ almost everywhere on the interval $(r_1, \xi)$.
In this case, we assume by definition that
$
    \psi (\xi)
    \varkappa^{
        \alpha - 1
    }
    (\xi)
    =
    0.
$
\end{Remark}

Lemma~\ref{L3.2} is proved in~\cite[Lemma~2.1]{meJMS}.

\begin{Lemma}\label{L3.3}
Let $w$ be a solution of~\eqref{1.1}, \eqref{1.2} and, moreover, $a \le r_1 < r_2$ be real numbers
such that $\sigma r_1 \ge r_2$, $w (r_1) > 0$, and 
$\theta^{1 / 2} w (r_2) \ge w (r_1) \ge \theta^{1 / 4} w (r_2)$.
Then
\begin{equation}
	\int_{
		w (r_2)
	}^{
		w (r_1)
	}
	\eta^{- 1 / m} (t)
	t^{1 / m - 1}
	\,
	dt
	\ge
	c
	\int_{r_1}^{r_2}
	\xi^{m - 1}
	f (\xi)
	\varphi^{1 / m - 1} (\xi)
	\,
	d\xi.
	\label{L3.3.1}
\end{equation}
\end{Lemma}

\begin{proof}
From Lemma~\ref{L3.1}, it follows that
$$
	w^{1 / m} (r_1)
	\inf_{
		[w (r_2), w (r_1)]
	}
	\eta^{- 1 / m}
	\ge
	c
	\left(
	\int_{r_1}^{r_2}
	(\xi - r_1)^{m - 1}
	f (\xi)
	\,
	d\xi
	\right)^{1 / m}.
$$
In so doing, by Lemma~\ref{L3.2}, we have
$$
	\left(
	\int_{r_1}^{r_2}
	(\xi - r_1)^{m - 1}
	f (\xi)
	\,
	d\xi
	\right)^{1 / m}
	\ge
	c
	\int_{
		r_1
	}^{
		r_2
	}
	(\xi - r_1)^{m - 1}
	f (\xi)
	\varkappa^{1 / m - 1}
	(\xi)
	\,
	d\xi,
$$
where
$$
	\varkappa (\xi)
	=
	\int_{r_1}^\xi
	(\zeta - r_1)^{m - 1}
	f (\zeta)
	\,
	d \zeta.
$$
Therefore,
\begin{equation}
	w^{1 / m} (r_1)
	\inf_{
		[w (r_2), w (r_1)]
	}
	\eta^{- 1 / m}
	\ge
	c
	\int_{
		r_1
	}^{
		r_2
	}
	(\xi - r_1)^{m - 1}
	f (\xi)
	\varkappa^{1 / m - 1}
	(\xi)
	\,
	d\xi.
	\label{PL3.3.1}
\end{equation}
Let us estimate the right-hand side of the last inequality.
Since
$$
	\frac{
		1
	}{
		(\xi - r_1)^m
	}
	\int_{r_1}^\xi
	(\zeta - r_1)^{m - 1}
	f (\xi)
	\,
	d \zeta
	\le
	\esssup_{
		(r_1, \xi)
	}
	f
	\le
	\xi^{-m}
	\varphi (\xi)
$$
for all $\xi \in (r_1, r_2)$, we obtain
$$
	(\xi - r_1)^{m - 1}
	\varkappa^{1 / m - 1}
	(\xi)
	=
	\left(
		\frac{
			1
		}{
			(\xi - r_1)^m
		}
		\int_{r_1}^\xi
		(\zeta - r_1)^{m - 1}
		f (\xi)
		\,
		d \zeta
	\right)^{1 / m - 1}
	\ge
	\xi^{m - 1}
	\varphi^{1 / m - 1}
$$
for all $\xi \in (r_1, r_2)$.
Hence, one can claim that
$$
	\int_{
		r_1
	}^{
		r_2
	}
	(\xi - r_1)^{m - 1}
	f (\xi)
	\varkappa^{1 / m - 1}
	(\xi)
	\,
	d\xi
	\ge
	\int_{
		r_1
	}^{
		r_2
	}
	\xi^{m - 1}
	f (\xi)
	\varphi^{1 / m - 1}
	\,
	d\xi.
$$
At the same time, for the left-hand side of~\eqref{PL3.3.1} we have
$$
	\int_{
		w (r_2)
	}^{
		w (r_1)
	}
	\eta^{- 1 / m} (t)
	t^{1 / m - 1}
	\,
	dt
	\ge
	c
	\,
	w^{1 / m} (r_1)
	\inf_{
		[w (r_2), w (r_1)]
	}
	\eta^{- 1 / m}.
$$
Thus, inequality~\eqref{PL3.3.1} implies~\eqref{L3.3.1}.

The proof is completed.
\end{proof}

\begin{Lemma}\label{L3.4}
Let $u : (0, \infty) \to (0, \infty)$ be a continuous function and, moreover,
$$
	v (t) 
	= 
	\inf_{
		(t / \lambda, \lambda t)
	}
	u,
	\quad
	t \in (0, \infty),
$$
where $\lambda > 1$ is some real number. Then
$$
	\left(
		\int_{t_1}^{t_2}
		v^{- 1 / m} (t)
		t^{1 / m - 1}
		\,
		dt
	\right)^m
	\ge
	A
	\int_{t_1}^{t_2}
	\frac{
		dt
	}{
		u (t)
	}
$$
for all non-negative real numbers $t_1$ and $t_2$ such that $t_2 \ge \lambda t_1$,
where the constant $A > 0$ depends only on $m$ and $\lambda$.
\end{Lemma}

Lemma~\ref{L3.4} is proved in~\cite[Lemma~2.3]{meIzv}.

\begin{proof}[Proof of Theorem~$\ref{T2.1}$]
Assume to the contrary that $w$ is a regular solution of problem~\eqref{1.1}, \eqref{1.2}.
In particular, $w$ is a positive function on the whole interval $[a, \infty)$.
This simple fact follows immediately from~\eqref{1.2}.

Construct a sequence of real numbers $\{ r_i \}_{i=0}^\infty$.
We take $r_0 = a$. Assume further that $r_i$ is already known.
If 
$
	\theta^{1/4} 
	w (\sigma^{1/2} r_i) 
	\ge 
	w(r_i),
$
then we put $r_{i+1} = \sigma^{1/2} r_i$; otherwise we take 
$r_{i+1} \in (r_i, \sigma^{1/2} r_i)$
such that
$
	\theta^{1/4} 
	w (r_{i+1})
	=
	w (r_i).
$

Let $\Xi_1$ be the set of positive integers $i$ satisfying the condition 
$\sigma^{1/2} r_{i-1} > r_i$
and $\Xi_2$ be the set of all other positive integers.
By Lemma~\ref{L3.3}, the inequality
\begin{equation}
	\int_{
		w (r_{i+1})
	}^{
		w (r_{i-1})
	}
	\eta^{- 1 / m} (t)
	t^{1 / m - 1}
	\,
	dt
	\ge
	c
	\int_{
		r_{i-1}
	}^{
		r_{i+1}
	}
	\xi^{m - 1}
	f (\xi)
	\varphi^{1 / m - 1} (\xi)
	\,
	d\xi
	\label{PT2.1.1}
\end{equation}
is fulfilled for all $i \in \Xi_1$. 
In turn, if $i \in \Xi_2$, then
$$
	w (r_{i-1})
	-
	w (r_{i+1})
	\ge
	c
	\sup_{
		[w (r_{i+1}), w (r_{i-1})]
	}
	\eta
	\int_{
		r_{i-1}
	}^{
		r_{i+1}
	}
	(\xi - r_{i-1})^{m - 1}
	f (\xi)
	\,
	d \xi
$$
according to Lemma~\ref{L3.1}. Combining this with the evident inequalities
$$
	\int_{
		w (r_{i+1})
	}^{
		w (r_{i-1})
	}
	\frac{
		dt
	}{
		\eta (t)
	}
	\ge
	(
		w (r_{i-1}) 
		-	
		w (r_{i+1})
	)
	\inf_{
		[w (r_{i+1}), w (r_{i-1})]
	}
	\frac{1}{\eta}
$$
and
$$
	\int_{
		r_{i-1}
	}^{
		r_{i+1}
	}
	(\xi - r_{i-1})^{m - 1}
	f (\xi)
	\,
	d \xi
	\ge
	c
	\int_{
		r_i
	}^{
		r_{i+1}
	}
	\xi^{m - 1}
	f (\xi)
	\,
	d \xi,
$$
we conclude that
\begin{equation}
	\int_{
		w (r_{i+1})
	}^{
		w (r_{i-1})
	}
	\frac{
		dt
	}{
		\eta (t)
	}
	\ge
	c
	\int_{
		r_i
	}^{
		r_{i+1}
	}
	\xi^{m - 1}
	f (\xi)
	\,
	d \xi.
	\label{PT2.1.2}
\end{equation}
By~\eqref{T2.1.2}, at least one of the following two relations is valid:
\begin{equation}
	\sum_{i \in \Xi_1}
	\int_{
		r_i
	}^{
		r_{i+1}
	}
	\xi^{m - 1}
	f (\xi)
	\mu^{1 / m - 1} (\xi)
	\,
	d \xi
	=
	\infty,
	\label{PT2.1.3}
\end{equation}
\begin{equation}
	\sum_{i \in \Xi_2}
	\int_{
		r_i
	}^{
		r_{i+1}
	}
	\xi^{m - 1}
	f (\xi)
	\mu^{1 / m - 1} (\xi)
	\,
	d \xi
	=
	\infty.
	\label{PT2.1.4}
\end{equation}

In the case that~\eqref{PT2.1.3} holds, summing~\eqref{PT2.1.1} over all $i \in \Xi_1$, we obtain
$$
	\int_0^{w (a)}
	\eta^{- 1 / m} (t)
	t^{1 / m - 1}
	\,
	dt
	=
	\infty.
$$
This contradicts condition~\eqref{T2.1.1}.

Assume that~\eqref{PT2.1.4} is fulfilled.
In this case, summing~\eqref{PT2.1.2} over all $i \in \Xi_2$, we have
\begin{equation}
	\int_0^{w (a)}
	\frac{
		dt
	}{
		\eta (t)
	}
	=
	\infty.
	\label{PT2.1.5}
\end{equation}
At the same time, from Lemma~\ref{L3.4}, it follows that
$$
	\left(
		\int_0^1
		g^{- 1 / m} (t)
		t^{1 / m - 1}
		\,
		dt
	\right)^m
	\ge
	c
	\int_0^1
	\frac{
		dt
	}{
		\eta (t)
	}.
$$
The last inequality and~\eqref{PT2.1.5} imply~\eqref{T2.3.1}.
Thus, we again arrive at a contradiction with~\eqref{T2.1.1}.

The proof is completed.
\end{proof}

\begin{Lemma}\label{L3.5}
Let~\eqref{T2.2.2} hold and, moreover, $\lambda > 1$ be some real number. 
Then
$$
	\int_a^{
		\lambda r
	}
	\xi^{m - 1}
	f (\xi)
	\,
	d\xi
	\le
	A
	\int_a^{
	r
	}
	\xi^{m - 1}
	f (\xi)
	\,
	d\xi
$$
for all sufficiently large $r$,
where the constant $A > 0$ depends only on $\lambda$ and on the value of the limit 
in the left-hand side of~\eqref{T2.2.2}.
\end{Lemma}

\begin{Lemma}\label{L3.6}
In the hypotheses of Lemma~$\ref{L3.5}$, let~\eqref{T2.4.1} be fulfilled 
instead of~\eqref{T2.2.2}, then
$$
	\int_r^\infty
	\xi^{m - 1}
	f (\xi)
	\,
	d\xi
	\le
	A
	\int_{
		\lambda r
	}^\infty
	\xi^{m - 1}
	f (\xi)
	\,
	d\xi
$$
for all sufficiently large $r$,
where the constant $A > 0$ depends only on $\lambda$ and on the value of the limit 
in the left-hand side of~\eqref{T2.4.1}.
\end{Lemma}

Lemmas~\ref{L3.5} and~\ref{L3.6} are proved in~\cite[Lemmas~2.7 and~2.8]{meJMS}.

\begin{Lemma}\label{L3.7}
Let $w$ be a regular solution of~\eqref{1.1}, \eqref{1.2}. 
If relations~\eqref{T2.2.1} and~\eqref{T2.2.2} are valid, then
\begin{equation}
	\int_{
		w (r)
	}^{
		w (a)
	}
	g^{- 1 / m} (t)
	t^{1 / m - 1}
	\,
	dt
	\ge
	c
	\left(
		\int_a^r
		\xi^{m - 1}
		f (\xi)
		\,
		d\xi
	\right)^{1 / m}
	\label{L3.7.1}
\end{equation}
for all sufficiently large $r$.
\end{Lemma}

\begin{proof}
Consider the sequence of real numbers $\{ r_i \}_{i=0}^\infty$ and the sets $\Xi_1$ and $\Xi_2$
constructed in the proof of Theorem~\ref{T2.1}. It can be seen that
\begin{equation}
	\lim_{n \to \infty} w (r_n) = 0.
	\label{PL3.7.12}
\end{equation}
In fact, if
$$
	\lim_{n \to \infty} w (r_n) = w (\infty) > 0,
$$
then $i \in \Xi_2$ for all sufficiently large $i$. 
Summing~\eqref{PT2.1.2} over all $i \in \Xi_2$, we obtain
$$
	\int_{
		w (\infty)
	}^{
		w (a)
	}
	\frac{
		dt
	}{
		\eta (t)
	}
	\ge
	c
	\sum_{i \in \Xi_2}
	\int_{
		r_i
	}^{
		r_{i+1}
	}
	\xi^{m - 1}
	f (\xi)
	\,
	d \xi.
$$
From~\eqref{T2.2.1}, it follows that the right-hand side of the last expression is equal to infinity, 
whereas the left-hand side is bounded. This contradiction proves~\eqref{PL3.7.12}.

According to~\eqref{T2.2.2}, there exists $j \ge 1$ such that
\begin{equation}
	\varphi (\xi)
	=
	\xi^m
	\esssup_{
		(\xi / \sigma, \xi)
		\cap
		[a, \infty)
	}
	f
	\le
	c
	\int_a^\xi
	\zeta^{m - 1}
	f (\zeta)
	\,
	d \zeta
	\label{PL3.7.1}
\end{equation}
for all $\xi \ge r_j$. 

We denote 
$
	\Xi_{1, n}
	=
	\{
		1 \le i \le n
		:
		i \in \Xi_1
	\}
$
and
$
	\Xi_{2, n}
	=
	\{
		1 \le i \le n
		:
		i \in \Xi_2
	\},
$
$n = 1, 2, \ldots$.
Also take an integer $l \ge 1$ such that $\theta^{1/2} w (r_{n+1}) \le w (a)$ and
\begin{equation}
	\frac{1}{4}
	\int_a^{r_{n+1}}
	\xi^{m - 1}
	f (\xi)
	\,
	d\xi
	\ge
	\int_a^{r_{j+1}}
	\xi^{m - 1}
	f (\xi)
	\,
	d\xi
	\label{PL3.7.2}
\end{equation}
for all $n \ge l$.
In view of~\eqref{PL3.7.12} and~\eqref{T2.2.1}, such an integer $l$ obviously exists.

At first, let
\begin{equation}
	\sum_{
		i \in \Xi_{1,n}
	}
	\int_{
		r_i
	}^{
		r_{i+1}
	}
	\xi^{m - 1}
	f (\xi)
	\,
	d\xi
	\ge
	\frac{1}{2}
	\int_a^{r_{n+1}}
	\xi^{m - 1}
	f (\xi)
	\,
	d\xi
	\label{PL3.7.3}
\end{equation}
for some $n \ge l$.
Summing~\eqref{PT2.1.1} over all $i \in \Xi_{1,n}$, we obtain
$$
	\int_{
		w (r_{n+1})
	}^{
		w (a)
	}
	\eta^{-1 / m} (t)
	t^{1 / m - 1}
	\,
	dt
	\ge
	c
	\sum_{
		i \in \Xi_{1,n}
	}
	\int_{
		r_i
	}^{
		r_{i+1}
	}
	\xi^{m - 1}
	f (\xi)
	\varphi^{1 / m - 1} (\xi)
	\,
	d\xi.
$$
By~\eqref{PL3.7.1}, this implies the inequality
\begin{align}
	\int_{
		w (r_{n+1})
	}^{
		w (a)
	}
	\eta^{- 1 / m} (t)
	t^{1 / m - 1}
	\,
	dt
	\ge
	{}
	&
	c
	\left(
		\int_a^{r_{n+1}}
		\xi^{m - 1}
		f (\xi)
		\,
		d \xi
	\right)^{1 / m - 1}
	\nonumber
	\\
	&
	{\times{}}
	\sum_{
		i 
		\in 
		\Xi_{1,n}
		\setminus
		\Xi_{1,j}
	}
	\int_{
		r_i
	}^{
		r_{i+1}
	}
	\xi^{m - 1}
	f (\xi)
	\,
	d\xi.
	\label{PL3.7.5}
\end{align}
It is easy to see that
$$
	\sum_{
		i 
		\in 
		\Xi_{1,j}
	}
	\int_{
		r_i
	}^{
		r_{i+1}
	}
	\xi^{m - 1}
	f (\xi)
	\,
	d\xi
	\le
	\int_a^{r_{j+1}}
	\xi^{m - 1}
	f (\xi)
	\,
	d\xi;
$$
therefore, taking into account~\eqref{PL3.7.2} and~\eqref{PL3.7.3}, we have
\begin{align*}
	\sum_{
		i 
		\in 
		\Xi_{1,n}
		\setminus
		\Xi_{1,j}
	}
	\int_{
		r_i
	}^{
		r_{i+1}
	}
	\xi^{m - 1}
	f (\xi)
	\,
	d\xi
	&
	=
	\sum_{
		i 
		\in 
		\Xi_{1,n}
	}
	\int_{
		r_i
	}^{
		r_{i+1}
	}
	\xi^{m - 1}
	f (\xi)
	\,
	d\xi
	-
	\sum_{
		i 
		\in 
		\Xi_{1,j}
	}
	\int_{
		r_i
	}^{
		r_{i+1}
	}
	\xi^{m - 1}
	f (\xi)
	\,
	d\xi
	\\
	&
	\ge
	\frac{1}{4}
	\int_a^{r_{n+1}}
	\xi^{m - 1}
	f (\xi)
	\,
	d\xi.
\end{align*}
The last relation and~\eqref{PL3.7.5} imply that
$$
	\int_{
		w (r_{n+1})
	}^{
		w (a)
	}
	\eta^{- 1 / m} (t)
	t^{1 / m - 1}
	\,
	dt
	\ge
	c
	\left(
		\int_a^{r_{n+1}}
		\xi^{m - 1}
		f (\xi)
		\,
		d \xi
	\right)^{1 / m}.
$$
Since $g (t) \le \eta (t)$ for all $t \in (0, \infty)$, this yields the estimate
\begin{equation}
	\int_{
		w (r_{n+1})
	}^{
		w (a)
	}
	g^{- 1 / m} (t)
	t^{1 / m - 1}
	\,
	dt
	\ge
	c
	\left(
		\int_a^{r_{n+1}}
		\xi^{m - 1}
		f (\xi)
		\,
		d \xi
	\right)^{1 / m}.
	\label{PL3.7.6}
\end{equation}

Now, assume that
$$
	\sum_{
		i \in \Xi_{2,n}
	}
	\int_{
		r_i
	}^{
		r_{i+1}
	}
	\xi^{m - 1}
	f (\xi)
	\,
	d\xi
	\ge
	\frac{1}{2}
	\int_a^{r_{n+1}}
	\xi^{m - 1}
	f (\xi)
	\,
	d\xi
$$
for some $n \ge l$.
Summing~\eqref{PT2.1.2} over all $i \in \Xi_{2,n}$, we obtain
$$
	\int_{
		w (r_{n+1})
	}^{
		w (a)
	}
	\frac{
		dt
	}{
		\eta (t)
	}
	\ge
	c
	\int_{
		a
	}^{
		r_{n + 1}
	}
	\xi^{m - 1}
	f (\xi)
	\,
	d \xi.
$$
Combining this with the inequality
$$
	\left(
		\int_{
			w (r_{n+1})
		}^{
			w (a)
		}
		g^{- 1 / m} (t)
		t^{1 / m - 1}
		\,
		dt
	\right)^m
	\ge
	c
	\int_{
		w (r_{n+1})
	}^{
		w (a)
	}
	\frac{
		dt
	}{
		\eta (t)
	}
$$
which follows from Lemma~\ref{L3.4}, we arrive at~\eqref{PL3.7.6} once more.

For any $r \ge r_{l+1}$ there exists $n \ge l$ such that $r_{n+1} \le r < r_{n+2}$.
From Lemma~\ref{L3.5}, we have
$$
	\int_a^{r_{n+1}}
	\xi^{m - 1}
	f (\xi)
	\,
	d \xi
	\ge
	c
	\int_a^r
	\xi^{m - 1}
	f (\xi)
	\,
	d \xi.
$$
It can also be seen that
$$
	\int_{
		w (r)
	}^{
		w (a)
	}
	g^{- 1 / m} (t)
	t^{1 / m - 1}
	\,
	dt
	\ge
	\int_{
		w (r_{n+1})
	}^{
		w (a)
	}
	g^{- 1 / m} (t)
	t^{1 / m - 1}
	\,
	dt.
$$
Combining the last two estimates with~\eqref{PL3.7.6}, we obtain~\eqref{L3.7.1}. 

The proof is completed.
\end{proof}

\begin{proof}[Proof of Theorem~$\ref{T2.2}$]
If $w$ is a regular solution of~\eqref{1.1}, \eqref{1.2}, 
then in accordance with Lemma~\ref{L3.7} inequality~\eqref{L3.7.1} is fulfilled 
for all $r$ in a neighborhood of infinity. 
Passing in this inequality to the limit as $r \to \infty$, we arrive at the contradiction
with~\eqref{T2.1.1} and~\eqref{T2.2.1}.

The proof is completed.
\end{proof}

\begin{proof}[Proof of Theorem~$\ref{T2.3}$]
If $w$ vanishes in a neighborhood of infinity, then~\eqref{T2.3.2} is obvious.
In turn, if $w$ is a regular solution of~\eqref{1.1}, \eqref{1.2}, 
then condition~\eqref{T2.2.1} and estimate~\eqref{L3.7.1} of Lemma~\ref{L3.7} 
imply the relation
$$
	\lim_{r \to \infty} w (r) \to 0.
$$
Hence,~\eqref{T2.3.1} allows one to claim that
$$
	\int_{
		w (r)
	}^1
	g^{- 1 / m} (t)
	t^{1 / m - 1}
	\,
	dt
	\ge
	\frac{1}{2}
	\int_{
		w (r)
	}^{
		w (a)
	}
	g^{- 1 / m} (t)
	t^{1 / m - 1}
	\,
	dt
$$
for all sufficiently large $r$.
Combining this with~\eqref{L3.7.1}, we have
$$
	\int_{
		w (r)
	}^1
	g^{- 1 / m} (t)
	t^{1 / m - 1}
	\,
	dt
	\ge
	c
	\left(
		\int_a^r
		\xi^{m - 1}
		f (\xi)
		\,
		d\xi
	\right)^{1 / m}
$$
for all sufficiently large $r$.
To complete the proof, it remains to note that the last expression is equivalent to~\eqref{T2.3.2}. \end{proof}

\begin{proof}[Proof of Theorem~$\ref{T2.4}$]
As mentioned above, every regular solution of~\eqref{1.1}, \eqref{1.2} 
is a positive function on the whole interval $[a, \infty)$.
This is obvious in view of~\eqref{1.2}.
By Lemma~\ref{L3.6} and relation~\eqref{T2.4.1},
there exists a real number $\rho \ge a$ such that
\begin{equation}
	\int_\xi^\infty
	\zeta^{m - 1}
	f (\zeta)
	\,
	d \zeta
	\le
	c
	\int_{
		\sigma^{1 / 2} 
		\xi
	}^\infty
	\zeta^{m - 1}
	f (\zeta)
	\,
	d \zeta
	\label{PT2.4.6}
\end{equation}
and
\begin{equation}
	\varphi (\xi)
	=
	\xi^m
	\esssup_{
		(\xi / \sigma, \xi)
		\cap
		[a, \infty)
	}
	f
	\le
	c
	\int_\xi^\infty
	\zeta^{m - 1}
	f (\zeta)
	\,
	d \zeta
	\label{PT2.4.1}
\end{equation}
for all $\xi \ge \rho$.
Let $r \in [\rho, \infty)$.
Consider a sequence of real numbers $\{ r_i \}_{i=0}^\infty$ defined as follows.
We take $r_0 = r$.
Assume that $r_i$ is already known.
If 
$
	\theta^{1/4} 
	w (\sigma^{1/2} r_i) 
	\ge 
	w(r_i),
$
then we put $r_{i+1} = \sigma^{1/2} r_i$; otherwise we take 
$r_{i+1} \in (r_i, \sigma^{1/2} r_i)$
such that
$
	\theta^{1/4} 
	w (r_{i+1})
	=
	w (r_i).
$

As in the proof of Theorem~\ref{T2.1}, by $\Xi_1$ we mean the set of positive integers 
$i$ satisfying the condition 
$\sigma^{1/2} r_{i-1} > r_i$.
Also let $\Xi_2$ be the set of all other positive integers.
Thus, inequality~\eqref{PT2.1.1} is valid for all $i \in \Xi_1$, 
whereas~\eqref{PT2.1.2} holds for all $i \in \Xi_2$.
Summing~\eqref{PT2.1.1} over all $i \in \Xi_1$, we obtain
\begin{equation}
	\int_0^{
		w (r)
	}
	\eta^{-1 / m} (t)
	t^{1 / m - 1}
	\,
	dt
	\ge
	c
	\sum_{
		i \in \Xi_1
	}
	\int_{
		r_i
	}^{
		r_{i+1}
	}
	\xi^{m - 1}
	f (\xi)
	\varphi^{1 / m - 1} (\xi)
	\,
	d\xi.
	\label{PT2.4.2}
\end{equation}
Analogously, from~\eqref{PT2.1.2}, it follows that
\begin{equation}
	\int_0^{
		w (r)
	}
	\frac{
		dt
	}{
		\eta (t)
	}
	\ge
	c
	\sum_{
		i \in \Xi_2
	}
	\int_{
		r_i
	}^{
		r_{i+1}
	}
	\xi^{m - 1}
	f (\xi)
	\,
	d \xi.
	\label{PT2.4.3}
\end{equation}

If
\begin{equation}
	\sum_{
		i \in \Xi_1
	}
	\int_{
		r_i
	}^{
		r_{i+1}
	}
	\xi^{m - 1}
	f (\xi)
	\,
	d\xi
	\ge
	\frac{1}{2}
	\int_{r_1}^\infty
	\xi^{m - 1}
	f (\xi)
	\,
	d\xi,
	\label{PT2.4.4}
\end{equation}
then in accordance with~\eqref{PT2.4.1} and~\eqref{PT2.4.2} we have
\begin{align*}
	\int_0^{
		w (r)
	}
	\eta^{- 1 / m} (t)
	t^{1 / m - 1}
	\,
	dt
	&
	\ge
	c
	\left(
		\int_{r_1}^\infty
		\xi^{m - 1}
		f (\xi)
		\,
		d \xi
	\right)^{1 / m - 1}
	\sum_{
		i \in \Xi_2
	}
	\int_{
		r_i
	}^{
		r_{i+1}
	}
	\xi^{m - 1}
	f (\xi)
	\,
	d \xi
	\\
	&
	\ge
	c
	\left(
		\int_{r_1}^\infty
		\xi^{m - 1}
		f (\xi)
		\,
		d \xi
	\right)^{1 / m}.
\end{align*}
This implies the estimate
\begin{equation}
	\int_0^{
		w (r)
	}
	g^{- 1 / m} (t)
	t^{1 / m - 1}
	\,
	dt
	\ge
	c
	\left(
		\int_{r_1}^\infty
		\xi^{m - 1}
		f (\xi)
		\,
		d \xi
	\right)^{1 / m}.
	\label{PT2.4.5}
\end{equation}
In turn, if~\eqref{PT2.4.4} is not valid, then
$$
	\sum_{
		i \in \Xi_2
	}
	\int_{
		r_i
	}^{
		r_{i+1}
	}
	\xi^{m - 1}
	f (\xi)
	\,
	d\xi
	\ge
	\frac{1}{2}
	\int_{r_1}^\infty
	\xi^{m - 1}
	f (\xi)
	\,
	d\xi.
$$
Therefore,~\eqref{PT2.4.3} allows one to claim that
$$
	\int_0^{
		w (r)
	}
	\frac{
		dt
	}{
		\eta (t)
	}
	\ge
	c
	\int_{r_1}^\infty
	\xi^{m - 1}
	f (\xi)
	\,
	d \xi.
$$
Since
$$
	\left(
		\int_0^{w (r)}
		g^{- 1 / m} (t)
		t^{1 / m - 1}
		\,
		dt
	\right)^m
	\ge
	c
	\int_0^{w (r)}
	\frac{
		dt
	}{
		\eta (t)
	}
$$
according to Lemma~\ref{L3.4}, we again obtain~\eqref{PT2.4.5}.

Finally, by relation~\eqref{PT2.4.6}, estimate~\eqref{PT2.4.5} implies that
$$
	\int_0^{
		w (r)
	}
	g^{- 1 / m} (t)
	t^{1 / m - 1}
	\,
	dt
	\ge
	c
	\left(
		\int_r^\infty
		\xi^{m - 1}
		f (\xi)
		\,
		d \xi
	\right)^{1 / m},
$$
whence~\eqref{T2.4.2} follows immediately.

The proof is completed.
\end{proof}

\end{document}